\documentclass[10pt,a4paper]{article}
\usepackage[utf8]{inputenc}
\usepackage{amsfonts,amsmath,amsthm,amscd,amssymb,latexsym}
\usepackage[margin=1.2in]{geometry}
\usepackage{tikz}
\usepackage{tabularx,ragged2e,booktabs,caption}
\usepackage{lscape} 
\usepackage{makeidx}				
\usepackage{enumerate}               
\usepackage{titlesec}
\usepackage{color}
\usepackage{makeidx}				
\usepackage{appendix} 
\usepackage{longtable}
\usepackage{mathtools}
\usepackage{times}

\DeclareMathOperator{\Aut}{Aut}

\DeclareMathOperator{\End}{End}
\DeclareMathOperator{\Hull}{Hull}

\DeclareMathOperator{\Gr}{\text{Gr}}

\DeclarePairedDelimiter\ceil{\lceil}{\rceil}
\DeclarePairedDelimiter\floor{\lfloor}{\rfloor}

\newcommand{\sq}{\mathbin{\square}}
\theoremstyle{plain}
\newtheorem{theorem}{Theorem}[section] 
\newtheorem{lemma}[theorem]{Lemma}
\newtheorem{problem}[theorem]{Problem}
\newtheorem{proposition}[theorem]{Proposition}
\newtheorem{corollary}[theorem]{Corollary}
\newtheorem{definition}[theorem]{Definition}
\newtheorem{conjecture}{Conjecture}
\theoremstyle{definition}
\newtheorem{example}[theorem]{Example}
\newtheorem{remark}[theorem]{Remark}

\begin{document}

\title{Generating Sets of the Kernel Graph and the Inverse Problem in Synchronization Theory}
\author{Artur Schaefer\\
  {\small Mathematical Institute, University of St Andrews}\\
  {\small North Haugh, St Andrews KY16 9SS, UK}\\
  {\small  as305@st-andrews.ac.uk}}
\date{}
\maketitle
\begin{abstract}
 
This paper analyses the construction of the kernel graph of a non-synchronizing transformation semigroup and introduces the inverse synchronization problem. Given a transformation semigroup $S\leq T_n$, we construct the kernel graph $\text{Gr}(S)$ by saying $v$ and $w$ are adjacent, if there is no $f\in S$ with $vf=wf$. The kernel graph is trivial or complete if the semigroup is a synchronizing semigroup or a permutation group, respectively. The connection between graphs and synchronizing (semi-) groups was established by Cameron and Kazanidis \cite{pjc08}, and it has led to many results regarding the classification of synchronizing permutation groups, and the description of singular endomorphims of graphs. This paper, firstly, emphasises the importance of this construction mainly by proving its superior structure, secondly, analyses the construction and discusses minimal generating sets and their combinatorial properties, and thirdly, introduces the inverse synchronization problem. The third part also includes an additional characterization of primitive groups.
\end{abstract}

\section{Introduction}
%

The motivation for this research comes from Cameron and Kazanidis \cite{pjc08} who provided a revolutionary approach to tackle the synchronization problem. Synchronization theory is the study of transformation semigroups $S$ admitting the synchronization property; a transformation groups is synchronizing, if it contains a transformation (or map) of rank $1$ (size of its image). However, a permutation group $G$ is synchronizing, if for every singular transformation $t$ the semigroup $S=\langle G,t\rangle$ is synchronizing. So, the synchronization problem is to classify all synchronizing permutation groups; however, a secondary problem is to find all tuples $(G,t)$ such that the corresponding semigroup is synchronizing.

If a group is not synchronizing, then there is a map $t$ of minimal rank which is not synchronized by $G$. Maps of minimal ranks correspond to section-regular partitions, as shown by Neumann \cite{neumann}. Moreover, such partitions are uniform and pose an interesting combinatorial object by themselves. However, Cameron and Kazanidis introduced a graph theoretical approach towards this problem. They showed that if there is a map not synchronized, then there is a graph admitting this map as an endomorphism. This is described in more detail in the next theorem.

\begin{theorem}[Thm 2.4 \cite{pjc08}]\label{thm1}
 A group $G$ does not synchronize a transformation $t$, if and only if there is non-trivial graph $X$ with complete core such that $\langle G,t\rangle\leq \End(X)$. 
\end{theorem}

Using this theorem the study of synchronizing groups translates into the study of graph endomorphisms, and many groups were shown to be non-synchronizing that way. Moreover, this theorem reignited the study of graph endomorphisms (which was possibly motivated by their application to synchronization theory \cite{godsil11,huang14,huang15,schaeferhamminggraph}.

However, the proof of this theorem makes use a special graph construction which guarantees a complete core, that is the equality of clique and chromatic number. The kernel graph $\Gr(S)$ of a transformation semigroup $S$ on $n$ points has vertex set $\{1,...,n\}$, where two vertices are adjacent if their is no transformation in $S$ collapsing these two vertices. If $S$ is the endomorphism monoid of a graph $X$, then $\Gr(S)$ is called the hull of $X$.  

The graph kernel graph forms the link between synchronization theory and graph endomorphisms (for graphs with complete core); in particularly, synchronization theory of permutation groups can be seen as the study of endomorphisms of non-trivial hulls. Both areas are of special interest to many mathematicians, and thus, the purpose of this research is to identify the basic properties of the kernel graph and emphasise its unique features.

This research is divided into $7$ sections. Section \ref{section2} is introducing the construction of $\Gr(S)$ and the term ``hull'' of a graph $X$ (written $\Hull(X)$), which uses $S=\End(X)$. In particularly, we call a graph a hull if it is its own hull. We provide basic properties and well known results on these graphs, and introduce the derived graph $\Gr'(S)$. Then, Section \ref{section3} is going to discuss briefly the effect of $k$-colourings provided by this construction through the clique number, and compare it with other types of colourings. It turns out that the $k$-colourings are the superior choice of colourings, and that hulls are the superior graphs regarding endomorphisms. Afterwards, in Section \ref{section4}, we present examples of both well known and less known graphs which are hulls and examples which are not hulls. Here, we show that one feature of rank $3$ graphs admitting singular endomorphisms is that they are hulls. In Section \ref{section5}, we are going to discuss generating sets of hulls, which turn out to contain interesting combinatorial properties. As a result, we will see that the idempotent transformations of minimal rank are sufficient to generate the hull. In Section \ref{section6} we will introduce the inverse synchronization problem, and provide a new characterization of primitive permutation groups. Finally, Section \ref{section7} provides a set of open problems.

%

\section{The Construction of the Kernel Graph and the Hull}\label{section2}

 In this section, we provide the construction of the kernel graph $\Gr(S)$ for a semigroup $S$ and define the term \textit{hull} for a graph $X$ (initially given in \cite{pjc08}). For a transformation semigroup $S\leq T_n$ the graph $\Gr(S)$ whose vertex set is $\{1,...,n\}$ where two vertices $v$ and $w$ are adjacent, if there is no $f\in S$ with $vf=wf$, is called the \emph{kernel graph}. This graph is admitting the following basic properties.
 
 \begin{lemma}
 Let $\Gamma=\Gr(S)$. Then, the following holds,
  \begin{enumerate}
   \item $S\leq \End(\Gamma)$,
   \item $\Gamma$ has clique number equal to its chromatic number,
   \item if $S$ is synchronizing, then $\Gamma $ is the null graph, and
   \item if $S$ is a permutation group, then $\Gamma$ is the complete graph.
  \end{enumerate}
 \end{lemma}
 \begin{proof}
  Everything except for $2$ is trivial. So, pick an element $t$ of minimal rank $r$ in $S$. The image of $t$ is a clique of size $r$ in $\Gamma$, since $t$ is minimal. But this means $t$ is a homomorphism from $\Gamma$ to the complete graph $K_r$ which means $t$ is a colouring.
 \end{proof}
 
 Applying this construction to the endomorphism monoid of a graph $X$ brings us to the definition of a hull.
 
 \begin{definition}
  Let $X$ be a graph and $E=\End(X)$. The \emph{hull} of $X$, denoted by $\Hull(X)$, is the graph $\Gr(E)$. If $X=\Hull(X)$, then we call $X$ a hull.
 \end{definition}
 
 From this construction the following properties are obvious.
 
 \begin{lemma}
 Let $X$ be a graph and $Y=\Hull(X)$. Then,
  \begin{enumerate}
   \item $X$ is a spanning subgraph of $Y$,
   \item $Y$ has clique number equal to chromatic number,
   \item $\Aut(X)\leq \Aut(Y)$, and
   \item $\End(X)\leq \End(Y)$.
  \end{enumerate}
  And in particular,
  \begin{enumerate}
   \item the null graph is a hull,
  \item the complete graph is a hull, and
  \item $\Hull(X)=\Hull(\Hull(X))$.
  \end{enumerate}

 \end{lemma}

 The hull turns out to be completing $X$ or at least adding extra symmetry to $X$, as we see from the last lemma. Moreover, $Y$ has a complete core (chromatic number and clique number are the same) which immediately provides an endomorphism for $Y$ by passing the complete graph 
 \[Y\rightarrow K_r \rightarrow Y.\]

 This construction gave rise to Godsil's and Royle's definition of pseudo-cores \cite{godsil11} which are graphs having either no singular endomorphisms or whose singular endomorphisms are all colourings. Their research covers the study of endomorphisms of geometric graphs, that is graphs coming from partial geometries. 

 However, the hull construction was extended to a graph called the derived graph of $\Gr(S)$ \cite{pjc13}. Let $X$ be a graph with clique number equal to chromatic number; then, the derived graph $Y$ is the graph with the same vertex set as $X$ and whose edges are the edges of $X$ which are contained in a maximal clique. In particular, for a given semigroup $S$ the derived graph of $\Gr(S)$ will be denoted by $\Gr'(S)$. The following properties can be found in \cite{araujo15}.
 
 \begin{lemma}
  \begin{enumerate}
   \item $S$ contains a map of rank $1$, if and only if $\Gr'(S)$ is the null graph.
   \item\label{endomorphismsderivedgraph} $S\leq \End(\Gr(S))\leq \End(\Gr'(S))$.
   \item $\Gr'(S)$ has a complete core, with clique and chromatic number the same as for $\Gr(S)$.
   \item Every edge of $\Gr'(S)$ is contained in a maximum clique.
  \end{enumerate}
 \end{lemma}
 
 Note, strict inequality can hold in \ref{endomorphismsderivedgraph}. For instance if we take the graph given by a disjoint union of complete graphs of different sizes.
 
 One important connection between a hull and its derived graph is given if $S$ is a maximal non-synchronizing submonoid of $T_n$. 
 
 \begin{theorem}[Thm 5.1, \cite{pjc13}]
  Let $S$ be a maximal non-sychronizing submonoid of $T_n$; moreover, let $X=\Gr(S)$ and $Y=\Gr'(S)$. Then,
  \begin{enumerate}
   \item $\End(X)=\End(Y)=S$, and
   \item $X=\Hull(Y)$.
  \end{enumerate}
 \end{theorem}

 If on the other hand a hull is equal to its derived graph, then the coverse holds, too.
 
 \begin{theorem}[Thm 5.2, \cite{pjc13}]
  Let $X$ be a hull with endomorphism monoid $S=\End(X)$, that is $X=\Gr(S)$. Assume that $X=\Gr'(S)$, then $S$ is a maximal non-synchronizing submonoid of $T_n$.
 \end{theorem}
 
 This research will not cover the properties of the derived graph, but instead focus on the kernel graph and the hulls, since these constructions are more approachable. For instance, one of the properties of the kernel graph is that its clique number is equal to its chromatic number which provides endomorphisms, and many graphs were investigated for complete cores and their endomorphisms. However, this makes us wonder how graphs with incomplete core but still admitting singular endomorphisms fit into this picture? We will discuss this briefly in the next section.

\section{The Hull and Colourings}\label{section3}

Once again, Theorem \ref{thm1} has reignited the study of graph endomorphisms for many semigroup and graph theorists;  however, the theorem actually mentions graphs having complete cores, what about graphs with other types of cores? Would other colourings play a role then? Of course, many examples of graphs admitting singular endomorphisms with non-complete core exist (cf. Example \ref{examplenoncompletecore}), but how do these graphs fit into the picture? What would happen, if other types of colourings would be used, for example Kneser colourings or circular colourings? The kernel graph and its hull offer answers to these questions. 

Recall, if $S$ is a transformation semigroup on $n$ points, then the kernel graph $\Gr(S)$ is the graph with vertex set $\{1,...,n\}$ where two vertices $v$ and $w$ are adjacent, if there is \textbf{no} transformation $t\in S$ with $vt=wt$. If $S$ is the endomorphism monoid of a graph $\Gamma$, then $\Gr(S)$ is the hull of $\Gamma$.

Suppose a graph $\Gamma$ admits singular endomorphisms; then a hull $Y$ can be obtained admitting all the endomorphisms of $\Gamma$; in addition, $Y$ has complete core. So clearly, whenever there is a non-hull graph $\Gamma$, a hull admitting the endomorphisms of $\Gamma$ can be found. 

This argument can be used in both directions. On the one hand, if the goal is to analyse $\End(Y)$ of a hull $Y$ through its subsemigroups, then it might be convenient to look for graphs $\Gamma$ with $Y=\Hull(\Gamma)$. There might be a chance that $\End(\Gamma)$ is a proper subsemigroup of $\End(Y)$. On the other hand, there might be purposes where information about $\End(Y)$ is good enough, that is ignoring any subgraphs might be clever and save some work. For instance, determining almost synchronizing groups allows to ignore any non-hulls and to focus on the endomorphisms of hulls.

Another interesting question is, what happens if other types of colourings would be used instead of $k$-colourings? So, what endomorphisms occur for graphs not having chromatic number equal to clique number ($\chi=\omega$) for $k$-colourings, but instead $\chi_C=\omega_C$ for a circular colouring $C$, or Kneser colouring, or other types of colourings? Or in other words, what effect do endomorphisms of graphs with non-complete core have on synchronization? Well, the result is the same as above. Since colourings of $\Gamma$ are homomorphisms to a graph $C$ on $r$ vertices, endomorphisms can be obtained by simply composing homomorphisms as
\[\Gamma\rightarrow C \rightarrow \Gamma\]
or
\[\Gamma\rightarrow C \rightarrow K_r \rightarrow \Gamma.\]
Hence, again it depends on whether we want to ignore substructures or not. So, using different colourings does not lead to new insights, which means that hulls contain all information. Consequently, hulls have a superior structure, and for this it is of interest to know which graphs are hulls and which are not. Thus, the next section is dedicated to provide examples of hulls and non-hulls.

\section{Examples of Hulls and Non-Hulls}\label{section4}

\subsection{Rank 3 Graphs}


The (permutation) \emph{rank} of a transitive permutation group
$G$ acting on a set $\Omega$ is the number of orbits of $G$ on $\Omega\times \Omega$, the set of ordered
pairs of elements of $\Omega$. Equivalently, it is the number of orbits on $\Omega$ of
the stabiliser of a point of $\Omega$. Certainly, these orbits correspond to graphs; therefore, a \emph{rank $3$ graph} is a graph whose transitive automorphism group has permutation rank $3$.

If $|\Omega| > 1$, then the rank of $G$ is at least $2$, because
no permutation can map $(x,x)$ to $(x,y)$. A primitive group of rank~$2$
is doubly transitive (admitting only the trivial graphs), and thus the first non-trivial cases
are primitive groups of rank~$3$. The aim of this section is to prove the following result. 

\begin{theorem}\label{rank3graphs} 
Every rank $3$ graph with singular endomorphisms is a hull.
\end{theorem}

Although a complete classification of the primitive  groups of rank~$3$ is known (see \cite{kantor82,liebeck86,liebeck87}),
we do not use this. Also, we do not use the combinatorial properties of strongly regular graphs which admitted by rank $3$ graphs. (A graph is \emph{strongly regular} if the numbers $k$, $\lambda$, $\mu$ of neighbours of a vertex, an edge, and a non-edge respectively are independent of the chosen vertex, edge or non-edge. See~\cite{pjcdesignsbook} for the definition and properties of strongly regular graphs. It is well known that a group with permutation rank $3$ is contained in the automorphism group of a strongly regular graph.) All we need to prove this is the orbit structure of rank $2$ and rank $3$ groups.



%


\begin{proof}
 Let $X$ be the rank $3$ graph and $\Aut(X)$ its automorphism group and $X'=\Hull(X)$. Note, $\Aut(X)$ is $2$-closed. For $\Aut(X)\leq \Aut(X')$, the automorphism group of $X'$ has either rank $3$ or rank $2$. If $\Aut(X')$ would have rank $2$, then $X'$ would be the null graph or the complete graph, but the complete graph has no singular endomorphisms and the null graph is not a supergraph of $X$. Hence, $\Aut(X')$ has rank $3$ and, thus, it acts on $X$ by automorphisms. Consequently, $\Aut(X)=\Aut(X')$, for the primed one is $2$-closed. This means, $X'$ is either $X$ or its complement $\overline{X}$. However, $X'$ cannot be $\overline{X}$, since $X$ is a spanning subgraph of $X'$.
\end{proof}

As mentioned in the introduction, the motivation to this research comes from synchronization theory and in \cite{pjc08} the authors covered rank $3$ graphs in that background. Here will list some families of rank $3$ graphs from their paper which are hulls.

\begin{corollary}
 \begin{enumerate}
  \item The Square lattice graph $L_2(n)$, for $n\geq 3$, is a hull.
  \item The triangular graph $T(n)$, for $n\geq 5$, is a hull.
  \item The Paley graph $P(q)$, $q$ a prime power congruent to $1$ mod $4$ and a square, is a hull.
 \end{enumerate}
\end{corollary}

Further examples of rank $3$ graphs are given by line graphs of projective spaces \cite{pjc08}.

\subsection{Unions of Complete Graphs and Multi-partite Graphs}

This section covers multi-partite graphs and their complements the unions of complete graphs. We will briefly describe their endomorphisms and show that these graphs are hulls.

\paragraph{Unions of Complete Graphs}
The union of $n$ copies of the complete graph $K_r$ is denoted by $n.K_r$. This graph has $n\cdot r$ vertices, is disconnected, and its automorphism group is $S_r \wr S_n$ with the imprimitive wreath product action. Moreover, its endomorphism monoid is easy to calculate: a singular endomorphism maps simply one copy of $K_r$ to another one.

Now, we consider the case where the $n$ copies are complete graphs of distinct sizes. So, let $X$ be the graph
\[K_{r_1}.K_{r_2}.\cdots .K_{r_s},\]
with $r_1\geq r_2\geq \cdots \geq r_s$. Then, a singular endomorphism maps smaller complete graphs to bigger complete graphs, that is, it maps $K_j$ to $K_i$, for $i,j\in \{r_1,...,r_s\}$ and $i>j$. So, essentially, the singular endomorphims admit a simple structure as well.

\paragraph{Multi-partite Graphs}

The multi-partite graph is the complement of the previous graph; this graph plays a major role in mathematics. However, its endomorphism monoid has a much more complicated structure (at least as a semigroup).

Let $X$ be the multi-partite graph
\[\overline{K_{r_1}.K_{r_2}.\cdots .K_{r_s}}.\]
How do the singular endomorphisms look like? Well, a singular endomorphism collapses two vertices, if and only if they are in the same part; so, it behaves nicely, too. Although these easily described endomorphims provide a chaotic semigroup structure of the endomorphim monoid, they are fine enough to construct the kernel graph. 

\begin{lemma}
The union of complete graphs $K_{r_1}.K_{r_2}.\cdots .K_{r_s}$ and its complement the multi-partite graph $\overline{K_{r_1}.K_{r_2}.\cdots .K_{r_s}}$ are hulls, for any values of $r_i$.
\end{lemma}

\subsection{Unions of Cores and their Complements}

\paragraph{Unions of Cores}
The previous setting can be generalized by taking unions of a graph $Y$ where $Y$ is a core. So, let $\Gamma$ be the graph
\[Y.Y.\cdots .Y,\]
given by $n$ copies of $Y$; we will write $\Gamma=n.Y$. Like for $U(n,r)$, the singular endomorphism monoid and the hull of $\Gamma$ can be determined easily. The following two results are obvious.

\begin{proposition}
\begin{enumerate}
 \item Let $\Gamma$ be the graph from above, and $t$ an endomorphisms collapsing two of the factors $Y$ and fixing the others, then it holds \[\End(\Gamma)=\langle \Aut(\Gamma),t\rangle .\]
 \item If, in addition, $Y$ is a transitive graph of order $r$, then \[\Hull(n.Y)=\Hull(n.K_r)=n.K_r.\]
\end{enumerate}
\end{proposition}
\begin{proof}
 Like for $n.K_r$, the group $\Aut(\Gamma)$ is permuting vertices within each $Y$ and the factors $Y$. Since $Y$ is a core, an endomorphism of $\Gamma$ is mapping some factors $Y$ to other factors $Y$. Thus, the first result follows.
 
 For the second part, we need to show that if two vertices come from distinct factors $Y$, then there is an endomorphism collapsing these vertices. Clearly, there are endomorphisms mapping one factor $Y$ to another. However, since $Y$ is transitive, each vertex of the first factor $Y$ can be mapped to any vertex of the second factor $Y$.
\end{proof}

\begin{remark}
 The second part of the previous lemma is not true, if $Y$ is non-transitive. For instance, let $Y$ be the wheel graph on 6 vertices, that is $5$ vertices form a cycle and the $6$th vertex is adjacent to all the others. This graph is not regular; hence, not transitive. Then, $\Hull(3.\Gamma)$ is a non-regular graph, and thus not equal to $3.K_6$.
\end{remark}


The odd cycle graphs $C_{2n+1}$ form another well-known family of graphs which are cores. The following example shows a surprising relation between unions of odd cycles and unions of complete graphs.

\begin{example}\label{chap8examplehullofgraph}
 Let $\Gamma=3.C_{5}$ be the graph given by $3$ copies of $C_5$. Both graphs $\Gamma$ and $3.K_5$ generate the same hull, but $\End(\Gamma)$ has size $27,000$, whereas $\End(3.K_5)$ has size $46,656,000$.
\end{example}

\paragraph{The Complementary Graph}
Next, the complementary graph $\overline{\Gamma}$ is considered. Unlike for the multi-partite graph it turns out that not all graphs $\overline{\Gamma}$ admit singular endomorphisms. Take a look at the next example.

\begin{example}
 Let $\overline{\Gamma}$ be the graph $\overline{3.C_5}$. The complement of the cyclic graph $C_5$ has no proper endomorphisms and for this reason $\overline{\Gamma}$ has no proper endomorphisms.
\end{example}

\begin{proposition}
 The graph $\overline{Y}$ is a core if and only if $\overline{n.Y}$ is a core.
\end{proposition}
\begin{proof}
 Assume $\overline{Y}$ is a core. It is easy to construct a singular endomorphism of $\overline{n.Y}$ which restricted to $\overline{Y}$ is a singular endomorphism of $\overline{Y}$. Thus, $\overline{n.Y}$ has no proper endomorphisms. Conversely, an endomorphism of $\overline{Y}$ can be extended to an endomorphism of $\overline{n.Y}$ by collapsing vertices in each subgraph $\overline{Y}$ in the same way.
\end{proof}

 \subsection{Cycles, Paths and other Non-Hulls}
 
 \begin{lemma}
  Let $C_n$ be a cycle with $n\geq 5$.
  \begin{enumerate}
   \item If $n$ is odd, then the hull of $C_n$ is the complete graph.
   \item If $n$ is even, then the hull of $C_n$ is the complete bipartite graph.
  \end{enumerate}
 \end{lemma}
 \begin{proof}
  
  Odd cycles are cores, thus there are no endomorphisms collapsing edges. Even cycles can be coloured with $2$ colours red and blue, so there are endomorphisms collapsing all vertices with colour red and others collapsing vertices with colour blue. Hence, edges only appear between vertices with distinct colours.
 \end{proof}
 
Note, the even cycle is a transitive non-hull graph with complete core. So, not all graphs with complete core are hulls.
 
 \begin{lemma}
  Let $P_n$ be a path with $n\geq 5$. Then, the hull of $P_n$ is the complete bipartite graph with parts of size $\ceil{\frac{n}{2}}$ and $\floor{\frac{n}{2}}$.
 \end{lemma}
 \begin{proof}
  The same argument as for the even cycle holds.
 \end{proof}

%

\subsection{Hamming Graphs and related Graphs}

The Hamming graph $H(m,n)$ is the graph with vertex set $\mathbb{Z}_n^{m}$ where two vertices are adjacent if their Hamming distance is $1$.  Also, the Hamming graph is the cartesian graph product of $m$ copies of the complete graph $K_n$, namely, 
\[ K_n\sq K_n \sq \cdots \sq  K_n.\]
However, if we change the Hamming distance from $1$ to $n$, then the resulting graph is the categorial graph product of complete graphs, namely,
\[ K_n \times K_n \times\cdots \times  K_n.\]
In \cite{schaeferhamminggraph} the author investigated the singular endomorphims of both graphs and showed that they behave nicely. Moreover, he showed that the singular endomorphims of the Hamming graph admit a Latin square and Latin hypercube structure. Clearly, these graphs are hulls and so are their complements.

\begin{lemma}
 \begin{enumerate}
  \item The cartesian product $\sq_{i=1}^{n} K_r$ and its complement are hulls, for $n\geq 2$ and $r\geq 3$.
  \item The categorial product $\bigtimes\limits_{i=1}^{n} K_r$ and its complement are hulls, for $n\geq 2$ and $r\geq 3$.
 \end{enumerate}
\end{lemma}
\begin{proof}
 The singular endomorphims were determined in \cite{schaeferhamminggraph}; so these results follow immediately.
\end{proof}

Another modification would be to consider cartesian and categorial products of other graphs. And surprisingly, our main example for graphs which are not hulls is given by the cartesian product of odd cycles
\[C_n\sq C_n \sq \cdots \sq C_n.\]
The case for two factors is given by the following example.

\begin{example}\label{examplenoncompletecore}
 The graph $X=C_n\sq C_n$, for odd $n\geq 5$, has satisfies the following. Its endomorphism monoid contains $8n^{2}$ singular transformations and is generated by the transformation $t$ corresponding to the Latin square $L$ and its automorphism group. Hence, it is a submonoid of the endomorphism monoid of $H(2,n)$, that is $\End(X)\leq \End(H(2,n))$, but its hull is $\Hull(X)=\overline{H(2,n)}$.
 \begin{equation*}
  L=\begin{pmatrix}
   1&2&3& \cdots &n-1&n\\
   n&1&2& \cdots &n-2&n-1 \\
   n-1&n&1& \cdots &n-3&n-2\\
   \vdots & \vdots &\vdots & \ddots & \vdots & \vdots \\
   2&3&4& \cdots &n&1
  \end{pmatrix}
 \end{equation*}
\end{example}

\subsection{Orthogonal Array Graphs and Latin Squares}

Orthogonal array graphs form another class of strongly regular graphs which contain the Hamming graphs $H(2,n)$ and their complements. These graphs come from orthogonal arrays with $n^2$ levels, $r$ factors, strength $2$ and index $1$. 

An \emph{orthogonal array}\index{orthogonal array} with $n$ levels, $r$ factors, of strength $t$ and index $\lambda$, i.e. a $t-(n,r,\lambda)$ orthogonal array, is a $r\times \lambda n^{t}$ array whose entries come from a set with $n$ elements such that in every subset of $t$ rows, every $t$-tuple appears in exactly $\lambda$ columns. In particular, $OA(r,n)$ denotes an orthogonal array with $t=2$ and $\lambda=1$. Moreover, an orthogonal array graph $L_r(n)$ is a graph whose vertices are the $n^{2}$ columns of $OA(r,n)$ where two vertices are adjacent, if the corresponding columns have the same entry in one of the coordinates. All we need to know about these graphs can be found in \cite{godsilbook}.

A \emph{Latin square}\index{Latin square} is an $n\times n$ array with entries from an $n$-element set, such that every row and every column contains each entry precisely once. Moreover, two Latin squares are \emph{mutually orthogonal}\index{mutually orthogonal Latin squares}, if after superimposition each of the $n^{2}$ distinct tuples occurs once. An $OA(3,n)$ orthogonal array represents a Latin square, where the three rows given row number, column number and symbol in each of the $n^2$ cells of the square (cf. Figure \ref{figureOAandLatinsquare}). In general, a set of $r-2$ mutually orthogonal Latin squares (MOLS) can be identified with an orthogonal array $OA(r,n)$, for $r\geq 3$.

\begin{figure}[t]
\begin{center}
\[
\begin{pmatrix}
  1&1&1&2&2&2&3&3&3\\
  1&2&3&1&2&3&1&2&3\\
  2&3&1&3&1&2&1&2&3\\
\end{pmatrix}\leftrightarrow
\begin{pmatrix}
 2&3&1\\
  3&1&2\\
  1&2&3\\
\end{pmatrix}
\]

 \end{center}
 \caption{Correspondence of OA(3,3) and a Latin square.}\label{figureOAandLatinsquare}
\end{figure}


By a result of Roberson \cite{davidroberson} the orthogonal array graphs are pseudo-cores. Moreover, it is well knon that an orthogonal array $OA(r,n)$ is extendible to $OA(r+1,n)$, if and only if $L_r(n)$ admits an $n$-colouring \cite[Thm 10.4.5]{godsilbook}. So, $L_r(n)$ admits singular endomorphims if and only if the corresponding orthogonal array is extendible. In some cases it is even possible to find a chain of extensions
\[OA(r,n) \rightarrow OA(r+1,n) \rightarrow \cdots \rightarrow OA(s-1,n) \rightarrow OA(s,n).\]

These $s-r$ extensions form a new orthogonal array $OA(s-r,n)$. If $s=n-1$ or $n-2$, then these extensions induce a complete set of $n-1$ MOLS. In that case, the orthogonal array graph $L_r(n)$ is a hull; however, the case where $OA(n,r)$ is not extendible to $OA(n-2,r)$ is unknown to the author. It follows:

\begin{lemma}
 Let $L_r(n)$ be an orthogonal array graph corresponding to a set of $r-2$ MOLS. If this set can be extended to a complete set of $n-1$ MOLS, then $L_r(n)$ is a hull.
\end{lemma}
\begin{proof}
 Let $v$ and $w$ two vertices of $L_r(n)$. Then, none of the $r-2$ Latin squares included in the orthogonal array has the same entry in the position of $v$ and $w$. However, since this set can be extended to a complete set of $n-1$ MOLS, there is one Latin square having the same entry in these positions. Thus, there is an endomorphism, induced by this Latin square, collapsing $v$ and $w$.
\end{proof}

\begin{corollary}
 Let $L_r(n)$ be an orthogonal array graph given by the desarguesian plane construction. Then, this graph is a hull.
\end{corollary}

\subsection{Small Primitive Graphs}
 
  In \cite{araujo15}, a complete list of small primitive graphs with singular endomorphisms on less than $45$ vertices is given. However, checking the graphs with $45$ vertices, we were able to determine that all graphs with less than $46$ vertices are hulls, except for one. This one graph has $25$ vertices and is the cartesian product of two $5$-cycles, namely $C_5\sq C_5$. We have already mentioned this construction above. Note, there are no primitive graphs on $46,47$ or $48$ vertices having singular endomorphisms; hence, it follows:
  
  \begin{theorem}
   All primitive graphs with singular endomorphisms and less than $49$ vertices, except for $C_5\sq C_5$, are hulls.
  \end{theorem}

\section{Generating Sets for $\Gr(S)$}\label{section5}

\subsection{Basic Results on Generating Sets}
When determining the examples form the previous section it was mostly argued that a graph is a hull, because all its singular endomorphisms are known. However, in this section we are interested deciding whether a graph is a hull by considering only a few singular endomorphisms, not all of them. That is, we would like to have generating sets for the kernel graph $\Gr (S)$, where $S$ is a semigroup. Mostly, there is no need to determine all endomorphisms, and the restriction to generating sets simplifies computations dramatically.

So, the question tackled in this section is: what are generators of the hull of $\Gamma$? Or more precisely, can we find a subset $S\subseteq \End(\Gamma)$ with $\Hull(\Gamma)=\Gr (S)$? The result is that we can choose a generating set which forms a left-zero semigroup (see Theorem \ref{theoremongeneratingsubsemigroup}). 

Recall, two vertices $v$ and $w$ in $\Gr (S)$ are adjacent, if there is no transformation $f\in S$ having $v$ and $w$ in the same kernel class. So, in fact, this construction is all about the kernels and we obtain the following observation, immediately. 

\begin{lemma}
 Let $S$ be a semigroup and $f_1,...,f_n$ representatives of its $R$-classes. Then, $\Gr (S)=\Gr (\{f_1,...,f_n\})$.
\end{lemma}

This result is one of the most important ones, regarding generating sets for $\Gr (S)$ and $\Hull(\Gamma)$, since it reduces the the number of generators to a generally much smaller set $\{f_1,...,f_n\}$. Another interesting result comes from an observation of non-synchronizing semigroups with non-trivial group of units $G$. The elements of minimal rank play an important role, since they have many important properties, and it turns out that these are sufficient to generate $\Gr (S)$. Note, the set of elements of minimal rank form the minimal ideal $I$ of $S$.

\begin{lemma}
 Let $S$ be a (non-synchronizing) semigroup of singular transformations and $I$ its minimal ideal, then $\Gr (I)=\Gr (S)$.
\end{lemma}
\begin{proof}
 Let $f_2=f_1t$ (right action), where $f_1,f_2\in T_n\setminus S_n$ and $t\in T_n$ are transformations, and let $n>rank(f_1)>rank(f_2)$.
If the transformation $f_1$ collapses the vertices $v$ and $w$, then so does $f_2$. Hence, if there is no such transformation $f_2$, then there is no such transformation $f_1$. Thus, it is enough to check the minimal ideal for adjacency.
\end{proof}

Assuming that $S$ has a transitive group of units, the transformations in $I$ need to be uniform \cite{neumann}. Hence, in these cases the construction of $\Gr(S)$ is based on a set of uniform partition. 

Note, the minimal ideal $I$ is a simple semigroup and a completely regular one \cite{howie}. Thus, every $H$-class of $I$ forms a group where its unique idempotent acts as the identity. This fact provides very interesting generating sets (at least for semigroup theorists).

\begin{corollary}\label{corminimalidealandidempotents}
 \begin{enumerate}
  \item Let $I$ be the minimal ideal of $S$ and $e_1,...,e_n$ representatives of the $R$-classes of $I$. Then, $\Gr(S)=\Gr(\{e_1,...,e_n\})$.
  \item The set of idempotents of $I$, namely $E(I)$, generates $\Gr(S)$.
  \item The set of idempotents of $S$ (here $E(S)$ respectively) generates $\Gr(S)$.
 \end{enumerate}
\end{corollary}
\begin{proof}
 The first result is a combination of the previous two. From this, the second follows, since the $H$-class cover the $R$-classes and, therefore, the idempotents contain a set of representatives from each $R$-class. The third result follows from the second.
\end{proof}

Combining these results, we are able to pick the right generators and chose any generating set we like. For instance, we can chose the generating set to be a left-zero semigroup.

\begin{theorem}\label{theoremongeneratingsubsemigroup}
 Let $S$ be a transformation semigroup with graph $\Gr(S)$. Then, we can find a left-zero semigroup $S'$ such that $\Gr(S)=\Gr(S')$. Moreover, the transformations of $S'$ are of minimal rank in $S$.
\end{theorem}
\begin{proof}
 As we have seen, the idempotents in $E(I)$ generate $\Gr(S)$. Pick an $L$-class $l$ in the minimal ideal $I$ of $S$. Then, the idempotents in $S'=E\cap l$ cover each $R$-class; hence this subset of idempotents generates $\Gr(S)$. However, this set $S'$ forms a left-zero semigroup, as can easily be checked. 
\end{proof}

Now, given that the first results on generating sets are established, they are applied to analyse semigroups of some particular types. In abstract semigroup theory many types of semigroups are common, for instance ``simple semigroups'', ``regular semigroups'' or even ``inverse semigroups''. Here, the following 4 kinds of semigroups are considered: monogenic semigroups, bands, semilattices and left-zero semigroups (right-zero semigroups, respectively). If $S$ is a semigroup of one of these types with an arbitrary transformation representation, then we want to know the generating set for $\Gr (S)$?

\paragraph{Monogenic Semigroups} A monogenic semigroup is the equivalent to a cyclic group in group theory. Here, the semigroup $S$ is generated by a single transformation $a$, namely $S=\langle a\rangle$, where $a^{m}=a^{m+r}$ for minimal non-negative integers $m$ and $r$. The integer $m$ is the index and $r$ is the period.

\begin{lemma}
 $\Gr (S)=\Gr (\{a^{m}\})$.
\end{lemma}
\begin{proof}
 The minimal ideal of $S$ is $I=\{a^{m},...,a^{m+r-1}\}$, where $I$ has only one $R$-class. The result follows from Theorem \ref{theoremongeneratingsubsemigroup}.
\end{proof}

\paragraph{Bands} A band is a semigroup $S$ where every element is an idempotent; that is, $a^2=a$, for all $a\in S$.

\begin{lemma}
 Let $S$ be a band and $I$ its minimal ideal. Further, let $b_1,...,b_s$ be a generating set for $I$. Then, $\Gr (S)=\Gr (\{b_1,...,b_s\})$.
\end{lemma}
\begin{proof}
 An element $x\in I$ is a word in the generators $b_1,...,b_s$. Thus, if the word starts with $b_i$, then $x$ and $b_i$ need to have the same kernel, as they already have the same rank. Hence, $b_1,...,b_s$ generate $\Gr (S)$.
\end{proof}

\paragraph{Semilattices} A semilattice is a semigroup which is a commutative band. Thus, we have $a^2=a$ and $ab=ba$, for all $a,b\in S$.

\begin{lemma}
 Let $S$ be a band and $I$ its minimal ideal. Further, let $b_1,...,b_s$ be a generating set for $I$. Then, $\Gr (S)=\Gr (\{b_1\})$.
\end{lemma}
\begin{proof}
 Since $S$ is a band, $\Gr (S)$ is generated by $b_1,...,b_s$. By the same argument as in the proof for bands, an element $x$ has the same kernel as $b_i$, for some $i$, and commutativity guarantees that $x$ has the same kernel as all the $b_i$. Therefore, we only need one of them to generate $\Gr (S)$.
\end{proof}

\paragraph{Left-zero Semigroups} A left-zero semigroup $S$ satisfies the following condition: \linebreak[4]$ab=a$, for all $a,b\in S$. In particular, left-zero semigroups are bands; however, generating sets for $\Gr (S)$ are even easier to determine.

\begin{lemma}\label{lemmaonleftzerosemigroups}
 Let $S$ be generated by $a_1,...,a_r$. Then, $\Gr (S)=\Gr (\{a_1,...,a_r\})$.
\end{lemma}
\begin{proof}
 For left-zero semigroups holds $\langle a_1,...,a_r\rangle=\{a_1,...,a_r\}$.
\end{proof}

Same holds for right-zero semigroups.


\subsection{Minimal Generating Sets}

As observed, the minimal ideal and, in fact, representatives of its $R$-classes already generate $\Gr (S)$. This reduces the size of a generating set significantly; however, the question of a minimal generating set arises. The interest in this question comes from its connections to graphs. In particular, if given a generating set for $\Gr (S)$, we can construct $\Gr (S)$, but also its complement graph, their automorphism group and their endomorphism monoids. All this information is implicitly included in the generators. Thus, the remainder of this section is devoted to minimal generating sets.

Unfortunately, this problem is not going to be solved in full generality, here, though it is easy to find minimal generating sets for the four kinds of semigroups from above. Instead, minimal generating sets are provided for a choice of the most important graphs from the examples in Section \ref{section4}. In detail, we cover the multi-partite graph, the ladder graph, the square lattice graph, the Hamming graph, and some of its variations. We start with the multi-partite graph, as this case is straightforward.

\begin{lemma}\label{chap8lemmamultipartitegraph}
 The multi-partite graph has a minimal generating set of size $1$.
\end{lemma}


Next, the ladder graph $LD(n,2)$ is considered. Here, we are going to encounter another famous combinatorial object; the binary Hamming code.

\begin{lemma}
 Let $n$ be a positive integer and let $r$ be minimal with respect to $n\leq | \mathbb{F}_2^{r}|$. Then, any $n$ vectors of $\mathbb{F}_2^{r}$ induce a generating set of size $r+1$ for $LD(n,2)$.
\end{lemma}
\begin{proof}
 We use a construction similar to the parity check matrix of the binary Hamming code. Let $M$ be a matrix whose columns are any $n$ vectors from $\mathbb{F}_2^{r}$. Now, add a row consisting of $1$'s to $M$. By substituting the $1$'s by the tuple $1,2$ and $0$'s by the tuple $2,1$, the rows of this matrix form transformations of $2n$ points. We can easily check that these transformations generate $LD(n,2)$.
\end{proof}

\begin{theorem}
 The hull of $LD(n,2)$ has a minimal generating set of size $r+1$, where $r$ is minimal with $n \leq |\mathbb{F}_2^{r}|$.
\end{theorem}
\begin{proof}
 We need to show that the above generating set is minimal. Assume we are given minimal generating set with $k<r+1$ elements. Wlog the images of these $k$ transformations are the set $\{1,2\}$. However, by the correspondence above (that is we encode the $1$'s and $0$'s as above) this leads to a $k\times n$ matrix whose columns are vectors in $\mathbb{F}_2^{k}$. Wlog we may assume that the last row consists of $1$'s, but then there is a column which appears twice in the matrix. Thus, there are too many edges between the $4$ vertices which are encoded by these two columns.
\end{proof}

It is more difficult to find minimal generating sets for $n.K_r$, where $r>2$ (which is a hull). However, it is possible to provide some bounds.

\begin{lemma}
 The graph $n.K_3$ can be generated by at most $n$ generators.
\end{lemma}
\begin{proof}
 Consider the $n\times n$ matrix  $M$, where $M$ has first row and first column $0$'s and the lower right $(n-1)\times (n-1)$ submatrix is $(J+I)$, where $J$ is the all $1$ matrix and $I$ the identity matrix. Now encode the $0$'s with the triple $1,2,3$; the $1$'s with $2,3,1$, and the $2$'s with $3,1,2$. The $n$ rows of the encoded $n\times 3n$ matrix generate $n.K_3$.
\end{proof}

\begin{lemma}
 The graph $n.K_r$ can be generated by at most $r$ generators, if $2\leq n \leq N(r)+1$ where $N(r)$ is the maximal number of mutually orthogonal Latin squares of order $r$.
\end{lemma}
\begin{proof}
 Let $2 \leq n\leq N(r)+1$ and consider the $r\times nr$ matrix $M=\left( A_1|A_2|\cdots |A_{n} \right)$ where $A_i$ are $r\times r$ matrices defined as follows: $A_1$ has all rows $1,2,...,r$; whereas, $A_2,...,A_{n}$ are the sets of mutually orthogonal Latin squares. The rows of $M$ form transformations on $nr$ points which generate $n.K_r$.
\end{proof}

We continue with the square lattice graph $L_2(n)$ and other Hamming graphs.

\begin{theorem}
 Let $\Gamma$ be the square lattice graph $L_2(n)$, which is a hull. Then, the following holds for generating sets of the hull:
 \begin{enumerate}
  \item If $n$ is a prime power, then the minimal generating set is given by a complete set of $n-1$ MOLS.
  \item If $n$ is no prime power, then the minimal generating set contains at most $n(n-1)$ elements.
 \end{enumerate}
\end{theorem}
\begin{proof}
 First, a complete set of $n-1$ MOLS generates $\Hull(L_2(n))$. If however, we would pick any $n-2$ transformations or less, then there would be too many edges in the resulting graph. 
 
 However, for non-prime power $n$, it is unknown whether there are complete sets of MOLS or not. In these cases, we pick the following transformations to generate the hull. Identify the vertices with points in $\mathbb{Z}_{n}^{2}$ and pick any Latin square. Fix one of its rows and permute the remaining $n-1$ cyclically using an $n-1$ cycle. Applying this permutation $n-1$ times results in $n-1$ distinct Latin squares. Doing this for all rows, provides us with $n(n-1)$ Latin squares, and thus, $n(n-1)$ transformations. It is clear that these generate the hull.
\end{proof}

This method can be easily extended to higher dimensional Hamming graphs.

\begin{corollary}
 Let $\Gamma$ be the Hamming graph $H(m,n)$, which is a hull. Then, the following holds for generating sets of the hull:
 \begin{enumerate}
  \item If $n$ is a prime power, then the minimal generating set is given by a complete set of orthogonal Latin hypercubes.
  \item If $n$ is no prime power, then the minimal generating set contains at most $n^{m-1}(n-1)^{m-1}$ elements.
 \end{enumerate}
\end{corollary}

Similarly, for orthogonal array graphs $L_k(n)$ coming from desarguesian affine planes the minimal generating set consists of the $n-k-1$ Latin squares which extend the initial set of MOLS to a complete set of $n-1$ MOLS. 


Now, further Hamming graphs are considered.

\begin{lemma}
 \begin{enumerate}
  \item The hull of $\overline{H(m,n)}$ has a minimal generating set of size $m$.
  \item The hull of $H(m,n;\{m\})$ has a minimal generating set of size $m$, too.
   \end{enumerate}
\end{lemma}
\begin{proof}
 In the first case, the $m$ transformations corresponding to the $m$ parallel class along the $m$ coordinate axes. In the second case, pick $m$ transformations each collapsing $(m-1)$-subarrays in one of $m$ possible ways. In both cases there cannot be less than $m$ transformations, since then we would not obtain the hull.
\end{proof}
%
%

In the next section, the inverse synchronization problem is introduced. Moreover, it is conjectured that every hull on $n$ vertices is generated by at most $n-1$ transformations.

\section{The Inverse Synchronization Problem}\label{section6}

The current research in synchronization theory is focused on the main problem in synchronization theory (the classification of synchronizing groups) and the secondary problem (of finding all tuples $(G,t)$ such that $\langle G,t\rangle$ is synchronizing) \cite{araujo13,pjc08,pjc13,araujo15,pjclectureonsynchronization}. The usual approach to these problems was through picking a group $G$ and finding all transformations not synchronized by it.

However, in this section we are going to reverse this approach, and change it to the problem of finding groups which do not synchronize a given set of maps. We call this problem the \textit{inverse synchronization problem}. The idea is the following. Given any set of maps $M$, construct the kernel graph $\Gr (M)$ to find its automorphism group $G$. 
\[M \rightarrow \Gr (M) \rightarrow \Aut(\Gr (M)).\]
The goal is to obtain an automorphism group not synchronizing the transformations in $M$, and to analyse it. However, this approach will not produce a satisfying result, in general, since there are things which can go wrong when considering a set $M$ instead of a semigroup $\langle M\rangle$. The next example provides a hint on what can go wrong. 

\begin{example}
 Consider the two transformations $t_1,t_2\in T_4$, where $t_1=[3,3,4,3]$ and $t_2=[3,3,2,4]$. The semigroup $S=\langle t_1,t_2\rangle$ contains a constant map $t=[4,4,4,4]$, therefore $\Gr (S)$ is the Null graph. However, if $M$ is the set $\{t_1,t_2\}$, then $\Gr (M)$ is non-trivial.
\end{example}

The reason for this discrepancy lies in the kernel structure of these transformations. Because, $t_2$ is a refinement of the kernel classes of $t_1$, the graph $\Gr (M)$ ignores $t_2$, that is, the kernel graph can be generated from $t_1$ alone. Therefore, semigroups need to be considered instead of sets. So, let $S$ denote the semigroup generated by the set $M$; then, the previous diagram transforms to 
\[S \rightarrow \Gr (S) \rightarrow \Aut(\Gr (S)).\]

By Theorem \ref{theoremongeneratingsubsemigroup} from the previous section, it can be assumed that $S$ is a left-zero semigroup. But then again, by the result on left-zero semigroups (Lemma \ref{lemmaonleftzerosemigroups}), $S$ can actually be taken to be a set; however, not just any set as seen from the last example. This means that, in fact, there are good choices and bad choices for picking a set $M$, as done initially.

Anyway, first, we consider the inverse synchronization problem for a single transformation (or respectively $S$ of size $1$); here, the dilemma of good and bad choices does not occur. Then, we discuss larger sets, and take a look at what groups can occur.

\subsubsection*{Semigroups with one Element}

Assume the semigroup $S$ contains a non-trivial singular transformation $t$ and has size $1$; so $t$ is an idempotent. However, how does the kernel graph $\Gr (S)$ look like? Well, two vertices are adjacent, if they are not in the same kernel class of $t$. Hence, the resulting graph is a multi-partite graph, each part corresponding to a kernel class of $t$.

These graphs have been covered in the examples section (see Section \ref{section2}) and by Lemma \ref{chap8lemmamultipartitegraph}, and an endomorphism of this graph is collapsing vertices lying in the same part. The structure of the automorphism group depends on the kernel structure, and it is imprimitive, in general; however, if $t$ is uniform, then at least transitivity holds. This provides the following characterization of primitivity. So, those groups are straightforward and the inverse synchronization problem is completely solvable.

\subsubsection*{Semigroups with more Elements}

In this section, the case with at least two generators for $\Gr (S)$ is considered. For this, $S$ needs to be a semigroup generated by at least two generators. So, what groups do not synchronize $S$? This question is really hard to solve, since in order to generate $\Gr (S)$ various combinations of kernel classes to need to be considered, in general. Hence, we are not able to provide an answer to this question, but rather provide a discussion and examples.

First, it is interesting to note the type of graphs which are generated by this construction. It is obvious, this construction generates graphs which are hulls; so, the non-synchronizing groups we obtain are automorphism groups of hulls. What about non-hulls? We need to leave this question and focus on automorphism groups of hulls.

So, because solving the inverse synchronization problem is very hard, it is of special interest to see what automorphism groups actually occur (or rather their isomorphism types) and which ones are likely to occur. As mentioned earlier, there are good and bad choices to pick a set $M$ of generators for $S$. Here, a bad choice is where $\Gr (S)$ provides a group which does synchronize some elements of $M$, but not all. A good choice is where a non-trivial group $G$ is obtained such that $\langle G,M\rangle$ is not-synchronizing. It appears that the bigger or the more structure the group $G$ has the better. So, we focus on the description of size and structure, where size is described quantitatively and structure qualitatively.

So, assuming a good set of transformations is chosen, which is providing a nice automorphism group. How good can this group be or, equivalently, how good can the given choice of generators in $M$ be? Are there choices which lead to hulls admitting a big (or well structured) non-synchronizing automorphism groups, and how many generators are needed to generate the corresponding graphs? These three questions are going to be tackled in the subsequent discussion.

Consider the first question: In what follows, choosing the transformations randomly is a bad choice. The reason for this is found in Cameron's paper \cite{pjc13}. He shows that by picking two random transformations of degree $n$ the semigroup generated by these transformations is synchronizing most of the time. Hence, $\Gr(S)$ is the null graph and its automorphism group is the whole symmetric group. This is a trivial answer to the inverse synchronization problem (that is, we obtain $S_n$). The result of Cameron is as follows.

\begin{lemma}
 The probability that two random transformations on $n$ points generate a synchronizing semigroup is about $1-O(n^{-2})$.
\end{lemma}

Clearly, this suggests that by picking a bigger set of random transformations the probability that the resulting semigroup is synchronizing increases. But what happens in one of the rare cases, if a non-synchronizing semigroup obtained? How good is a good (random) choice of generators? That is, how big or how structured is the group likely to be? Well, for instance the stabilizer of a point in $S_n$, that is $S_{n-1}$, occurs as an automorphism group. This group is the biggest possible non-trivial group which can occur, and a possible construction for its graph is given in the next example.

\begin{example}
 Let $2\leq k\leq n-1$ and $\Gamma$ be a graph on $n$ points given by the complete graph on $k$ points with $n-k$ extra vertices without edges. This graph has automorphism $S_{k}$ (which might permute the $n-k$ vertices in any possible way). Also, this graph is a hull, and a minimal generating set contains $k$ transformations $t_1,...,t_{k}$ where each transformation $t_i$ maps the $n-k$ points to the point $i$ and fixes the others.
\end{example}

The automorphism groups from the previous example are intransitive, but groups admitting a nicer structure can be found, as well. One example is the complement of the Hamming graph $H(2,n)$. From the previous section we know that its minimal generating set is of size two. Moreover, its automorphism group is the primitive group $S_n\wr S_2$ with permutation rank $3$. This group has a richer structure, but is smaller in terms of size ($|S_n\wr S_2|=(n!)^{2}$ compared to $|S_{n^{2}}|=(n^{2})!$ both on $n^{2}$ points). Another example is the complete multi-partite graph which has a transitive, but imprimitive automorphism group. Thus, the occurring groups vary fundamentally, and with the right choice of generators both large groups and groups with a rich structure can be obtained. This statement is underlined by Table \ref{tablewithisomorphismtypes}, where the isomorphism types of all occurring automorphism groups for very small $n$ is listed. As can be observed, many different structures occur.

\begin{table}[!t]
\begin{center}
 \begin{tabular}{lc|lc}\hline
  Vertices: $n=3$ & & $n=4$ & \\\hline
  \# Graphs & 4 & \# Graphs & 11 \\
  \# Hulls & 4 & \# Hulls & 10 \\\hline
  Groups & Occurrences &Groups & Occurrences \\\hline
    $C_2$ & 2 & $C_2$ & 2 \\
    $S_3$ & 2 & $C_2\times C_2$ & 2 \\
    &&$D_8$ & 2\\
    &&$S_3$ & 2\\
    &&$S_4$ & 2\\\hline
    Vertices: $n=5$ & & $n=6$ & \\\hline
  \# Graphs & 34 & \# Graphs & 156 \\
  \# Hulls & 27 & \# Hulls & 102 \\\hline
  Groups & Occurrences &Groups & Occurrences \\\hline
  $C_2$ & 5 & $\langle 1 \rangle$ & 3 \\
  $C_2\times C_2$ & 6 & $C_2$ & 22 \\
  $D_{12}$ & 6 & $C_2\times C_2$ & 21 \\
  $D_{8}$ & 4 & $C_2\times C_2\times C_2$ & 4 \\
  $S_3$ & 2 & $S_3$ & 4 \\
  $S_4$ & 2 & $D_8$ & 7 \\
  $S_5$ & 2 & $D_{12}$ & 17 \\
  &&$C_2\times D_8$ & 6\\
  &&$S_4$ & 2\\
  &&$S_3\times S_3$ & 2\\
  &&$C_2\times S_4$ & 8\\
  &&$(S_3\times S_3) \rtimes C_2$ & 2\\
  &&$S_5$ & 2\\
  &&$S_6$ & 2\\\hline
   Vertices: $n=7$ & & $n=7$ (continued) & \\\hline
  \# Graphs & 1044 & & \\
  \# Hulls & 539 & & \\\hline
  Groups & Occurrences &Groups & Occurrences \\\hline
 $\langle 1\rangle$ & 49 & $C_2\times D_8$ & 20 \\
 $C_2$ & 142 & $C_2\times S_4$ & 20 \\
 $D_8$ & 21 & $C_2\times S_5$ & 6 \\
 $S_3$ & 21 & $D_8\times S_3$ & 8 \\
 $S_4$ & 2 & $S_3\times S_3$ & 6 \\
 $S_5$ & 2 & $S_3\times S_4$ & 6 \\
 $S_6$ & 2 & $C_2\times C_2\times C_2$ & 29 \\
 $S_7$ & 2 & $C_2\times C_2 \times S_3$ & 18 \\
 $D_{12}$ & 47 & $(S_3\times S_3)\rtimes C_2$ & 4 \\
 $C_2\times C_2$ & 133 & & \\
 
 \end{tabular}

\end{center}

\caption{Distribution of isomorphism types of automorphism groups from small hulls.}\label{tablewithisomorphismtypes}
\end{table}

The third question is on the minimal number of generators for these hulls. Above it is mentioned that the more generators are picked randomly, the greater is the probability to obtain a synchronizing group, and thus, obtain the trivial answer. So, if the generators of $S$ would be picked randomly, the probability of getting a non-synchronizing group is decreasing each time an additional transformation is picked. Thus, how many need to be picked at most? Or equivalently, what are the sizes of the minimal generating sets of hulls on $n$ vertices? 

From the previous example, it can be observed that every number between $1$ and $n-1$ may occur, but it is unclear for bigger values. Our guess is that $n-1$ transformations are enough to generate any hull on $n$ vertices, and the data in Table \ref{tableofsizesofminimalgeneratingsets} supports this guess. This table contains the number of hulls having a minimal generating set of size $i$, for $i\in \mathbb{N}$, and as we see, the maximal size is $n-1$, indeed. We are missing a proof for this guess, but we conjecture the following.

\begin{conjecture}\label{chap8conjecture}
 A graph on $n$ vertices which is a hull can be generated with at most $n-1$ transformations.
\end{conjecture}

To conclude this section, we summarize the previous discussion. It was shown that by moving from semigroups generated by a single transformation to semigroups with more generators no satisfactory answer to the inverse synchronization problem was given. The difficulties lie in the vast number of possible outcomes of graphs $\Gr (S)$. For instance, from Table \ref{tableofsizesofminimalgeneratingsets} it can be observed that for $n=7$ a total of $112=15+97$ different graphs can be generated from only two transformations.

Moreover, it was pointed out that the occurring non-synchronizing automorphisms groups can be both well structured and big. Also, it is shown that there exist graphs on $n$ vertices which cannot be constructed with less than $i$ transformations, for any $i=1,...,n-1$, which leads to Conjecture \ref{chap8conjecture}.

\begin{table}
 \begin{center}
  \begin{tabular}{lc|lc}
   Vertices: $n=4$ & & $n=5$ & \\\hline
  \# Graphs & 11 & \# Graphs & 34 \\
  \# Hulls & 10 & \# Hulls & 27 \\\hline
  Size & Occurrences &Size & Occurrences \\\hline
    $1$ & 6 & $1$ & 7 \\
    $2$ & 2 & $2$ & 12 \\
    $3$ & 1 & $3$ & 7\\
    &&$4$ & 1\\\hline
   Vertices: $n=6$ & & $n=7$ & \\\hline
  \# Graphs & 156 & \# Graphs & 1044 \\
  \# Hulls & 102 & \# Hulls & 539 \\\hline
  Size & Occurrences &Size & Occurrences \\\hline
    $1$ & 11 & $1$ & 15 \\
    $2$ & 35 & $2$ & 97 \\
    $3$ & 46 & $3$ & 316\\
    $4$ & 9  & $4$ & 100\\
    $5$ & 1  & $5$ & 10\\
    &        & $6$ & 1\\\hline
  \end{tabular}
 \end{center}
\caption{Distribution of sizes of minimal generating sets of small hulls.}\label{tableofsizesofminimalgeneratingsets}
\end{table}

\section{Problems}\label{section7}

Another set of questions might come from the following problem.

\begin{problem}
 Let $X$ be a graph with $\Hull(X)=L_2(n)$. Is it true or false that $X=L_2(n)$?
\end{problem}

Considering the same question with $\overline{L_2(n)}$, the graph $X$ does not need to be $\overline{L_2(n)}$, but it could be the cartesian product of two odd cycles ($n$ respectively) (see Example \ref{examplenoncompletecore}).

\begin{problem}
 Find more examples of hulls and in particular of non-hulls.
\end{problem}

\begin{problem}
 Find minimal generating sets for
 \begin{enumerate}
  \item the union of complete graphs $n.K_r$, for $r>2$.
  \item for the triangular graph $T(n)$.
  \item other examples of hulls, in general.
 \end{enumerate}
\end{problem}

\begin{problem}
 Find a better approach to the inverse synchronization problem.
\end{problem}

\end{document}